\documentclass[11pt,a4paper]{article}
\usepackage[english]{babel}  
\usepackage[margin=1in]{geometry}
\RequirePackage[OT1]{fontenc}
\RequirePackage{graphicx}
\RequirePackage{amsthm,amssymb}
\RequirePackage[cmex10]{amsmath}
\RequirePackage[colorlinks,citecolor=blue,urlcolor=blue]{hyperref}

\usepackage[round,authoryear]{natbib}
\usepackage{graphicx,epstopdf}
\usepackage{listing}
\usepackage{amsfonts}
\usepackage{amsmath}		
\usepackage{bbm} 	
\usepackage{accents}
\usepackage{amsmath,bm}
\usepackage{mathtools}
\usepackage{booktabs}
\usepackage{xcolor}
\usepackage{amsmath}
\usepackage{amsfonts}
\usepackage{amssymb}
\usepackage{adjustbox}
\usepackage{amsmath, amssymb, array, adjustbox, booktabs}
\usepackage{booktabs}

\usepackage{caption} 
\usepackage{verbatim}
\usepackage{multirow}
\usepackage{authblk}
\usepackage[multiple]{footmisc}
\usepackage{easy-todo}

\usepackage{color}
\usepackage{array}
\definecolor{dkgreen}{rgb}{0,0.6,0}
\definecolor{gray}{rgb}{0.5,0.5,0.5}
\definecolor{mauve}{rgb}{0.58,0,0.82}

\usepackage{xr}

\makeatletter
\newcommand*\rel@kern[1]{\kern#1\dimexpr\macc@kerna}
\newcommand*\widebar[1]{%
  \begingroup
  \def\mathaccent##1##2{%
    \rel@kern{0.8}%
    \overline{\rel@kern{-0.8}\macc@nucleus\rel@kern{0.2}}%
    \rel@kern{-0.2}%
  }%
  \macc@depth\@ne
  \let\math@bgroup\@empty \let\math@egroup\macc@set@skewchar
  \mathsurround\z@ \frozen@everymath{\mathgroup\macc@group\relax}%
  \macc@set@skewchar\relax
  \let\mathaccentV\macc@nested@a
  \macc@nested@a\relax111{#1}%
  \endgroup
}
\makeatother

\newtheorem{assumption}{Assumption}

\newtheorem{proposition}{Proposition}

\newtheorem{remark}{Remark}
\newtheorem*{assumption*}{Assumption}

\newcommand{\defeq}{\equiv}

\newcommand{\trans}{\prime}              

\newcommand{\indicator}{\mathbbm{1}}     
\newcommand{\dd}{\mathrm{d}}             
\newcommand{\mfzero}{\boldsymbol{0}}         
\newcommand{\rR}{\mathbb{R}}             
\newcommand{\rN}{\mathbb{N}}             

\newcommand{\op}{o_{\mathrm{P}}}  
\newcommand{\Op}{O_{\mathrm{P}}}       
\newcommand{\filtr}{\mathcal{F}}
\newcommand{\filtrn}{\mathcal{F}^{(T)}}

\newcommand{\lawn}{\mathrm{P}^{(T)}}
\newcommand{\Exp}{\mathbb{E}}
\newcommand{\prob}{\mathrm{Pr}}

\newcommand{\experimentn}{\mathcal{E}^{(T)}}
\newcommand{\stat}{\tau}

\newcommand{\cond}{\,|\,}
\newcommand{\rate}{r}
\newcommand{\Rate}{\mathbf{R}}

\newcommand{\pint}{\theta}
\newcommand{\pintb}{\boldsymbol{\pint}}
\newcommand{\Pintb}{\boldsymbol{\Theta}}
\newcommand{\lpint}{h}
\newcommand{\lpintb}{\boldsymbol{\lpint}}

\newcommand{\LLR}{\Lambda}
\newcommand{\LLRn}{\LLR^{(T)}}

\newcommand{\measure}{\nu}
\newcommand{\freq}{D}

\newcommand{\vS}{\boldsymbol{S}}

\newcommand{\Y}{Y}
\newcommand{\Z}{Z}
\newcommand{\A}{A}

\newcommand{\Zkt}{\Z_{k,t}}

\newcommand{\ZAt}{\Z_{\A_t,t}}

\newcommand{\SK}{[K]}
\newcommand{\Sn}{[T]}

\newcommand{\e}{\varepsilon}

\newcommand{\ekt}{\e_{k,t}}
\newcommand{\f}{f}
\newcommand{\fzk}{\f_{k}}

\newcommand{\CS}{\Delta}
\newcommand{\CSb}{\boldsymbol{\CS}}
\newcommand{\CCSb}{\bm{\mathit{\Delta}}}
\newcommand{\QV}{\mathcal{Q}}
\newcommand{\QVb}{\boldsymbol{\QV}}
\newcommand{\QQVb}{\bm{\mathcal{J}}}
\newcommand{\FJ}{J}
\newcommand{\FJb}{\boldsymbol{\FJ}}

\newcommand{\FJbpintk}{\boldsymbol{\FJ}_{\pintb,k}}

\newcommand{\scorebk}{\dot{\boldsymbol{\ell}}_{\pintb,k}}

\graphicspath {{figures/}}

\title{Local Asymptotic Normality for Multi-Armed Bandits}
\author[1]{Ramon van den Akker}
\author[2]{Bas J.M.\ Werker}
\author[3]{Bo Zhou}
\affil[1]{Econometrics Group, Tilburg University}
\affil[2]{Econometrics and Finance Group, Tilburg University}
\affil[3]{Department of Economics, Virginia Tech}

\begin{document}

\setlength{\footnotesep}{12pt}

\maketitle


\abstract{
\noindent
\citet*{zhou2025bandit} showed that the limit experiment, in the sense of H\a'{a}jek-Le Cam, for (contextual) bandits whose arms' expected payoffs differ by $O(T^{-1/2})$, is Locally Asymptotically Quadratic (LAQ) but highly non-standard, being characterized by a system of coupled stochastic differential equations. The present paper considers the complementary case where the arms' expected payoffs are fixed with a unique optimal (in the sense of highest expected payoff) arm. It is shown that, under sampling schemes satisfying mild regularity conditions (including UCB and Thompson sampling), the model satisfies the standard Locally Asymptotically Normal (LAN) property.
}

\section{Introduction}
This paper considers the multi-armed bandit problem.
At each time step $t \in \Sn \defeq {1, \dots, T}$, an agent selects one of $K>1$ arms. Each arm $k\in\SK$ generates i.i.d.\ outcomes from an unknown distribution belonging to some parametric family. Let $\Zkt$ denote the $\rR$-valued \textit{potential} outcome of arm $k$ at time $t$. At step $t$, the agent only observes the chosen arm $\A_t$ and its realized outcome $\Y_t = \ZAt$.

We assume that $\Zkt$ follows the law $\mathcal{L}_{\pintb,k}$, where $\pintb\in\Pintb \subset \rR^p$ with $p\in\mathbb{N}$ and $\Pintb$ open. The parameter $\pintb$ indexes all arm distributions, although certain components of $\pintb$ may pertain only to specific arms. 

Throughout, we assume that there is a unique optimal arm (i.e., the arm with the highest expected payoff) and that the distributions $\mathcal{L}_{\pintb,k}$ are Differentiable in Quadratic Mean. We also impose regularity conditions on the adopted sampling policy, which are, for instance, satisfied by the popular Gaussian Thompson sampling and UCB-type policies. The precise assumptions are detailed in Section~\ref{sec:ass}.

Under these conditions, we show that the multi-armed bandit model satisfies the Locally Asymptotically Normal (LAN) property (see, e.g., \cite{van2000asymptotic}).
This stands in sharp contrast to the case where the arms' means are (only) 
$O(T^{-1/2})$ apart. For that setting (studied in, among others, \cite{kuang2024weak} and \cite{fan2025diffusion}),
\citet*{zhou2025bandit} demonstrated that the limit experiment, in the sense of H\a'{a}jek-Le Cam, is Locally Asymptotically Quadratic (LAQ) and highly non-standard, being characterized by a system of coupled stochastic differential equations.

The LAN property provides a classical asymptotic framework for analyzing efficiency bounds of estimators and tests and for developing optimal inference procedures. Nevertheless, a small Monte Carlo experiment (see Section~\ref{sec:MC}) suggests that, for moderate sample sizes and realistic parameter values, the asymptotic approximations of \citet*{zhou2025bandit} are often more accurate approximations to finite-sample behavior. This leads us to warn practitioners that relying on classical asymptotic distributional theory may be misleading in the settings studied in this paper.

The remainder of this paper is organized as follows. Section~\ref{sec:ass} gathers and discusses all assumptions on the arms' distributions and on the sampling strategy used. Our main convergence result is stated and proved in Section~\ref{sec:limitexperiment}. The results on our Monte Carlo experiment are provided in Section~\ref{sec:MC}.

\section{Assumptions}\label{sec:ass}
We assume that each arm's reward distribution $\mathcal{L}_{\pintb,k}$ admits a density $f_{k}(\cdot \cond \pintb)$, with respect to a common $\sigma$-finite dominating measure $\measure$. Furthermore, we impose the \textit{Differentiable in Quadratic Mean (DQM)} condition on these densities, along with a \textit{unique optimal arm} condition as described below.

\begin{assumption}\label{assm:DQM}
Let $\pintb \in \Pintb$ and $k \in \SK$. We assume the following conditions on the arm distributions.
\begin{itemize}
\item[(a)] The densities $\fzk$ are strictly positive and differentiable in quadratic mean at $\pintb$, that is,
\begin{align*}
\frac{\sqrt{\fzk(\Z_k \cond \pintb + \boldsymbol{\omega})}}{\sqrt{\fzk(\Z_k \cond \pintb)}} 
= 1 + \frac{1}{2}\left(\scorebk(\Z_k)^\trans \boldsymbol{\omega} + r_k(\Z_k \cond \boldsymbol{\omega})\right),
\end{align*}
for all $\boldsymbol{\omega} $ with $\pintb + \boldsymbol{\omega} \in \Pintb$, where $\scorebk(\cdot)$ is the $p$-dimensional score for arm $k$ satisfying $\Exp_{\pintb}\big[|\scorebk(\Z_{k})|^2\big] \in (0,\infty)$, and with $\Exp_{\pintb}\big[r_k^2(\Z_{k} \cond \boldsymbol{\omega})\big] = o(|\boldsymbol{\omega}|^2)$.
\item[(b)]
If the $j$-th component of the $p$-vector $\scorebk(Z_{k})$ is equal to $\mfzero$ a.s., then the
mapping $u\mapsto f_k(\cdot |\pintb(u))$ with 
$\pintb(u)=(\pint_1,\dots,\pint_{j-1}, u,\pint_{j+1},\dots,\pint_p)$,
is constant on an interval around $\pint_j$.
\item[(c)] Let $\mu_k(\pintb) \defeq \Exp_{\pintb}\big[\Z_{k}\big]$. There exists a unique $k^* = k^*_{\pintb} \in \SK$ such that
$ \mu_{k^*}(\pintb) > \max_{k \neq k^*} \mu_k(\pintb).
$
\end{itemize}
\end{assumption}

\begin{remark}
Assumption~\ref{assm:DQM}(a) implies $\Exp_{\pintb}\big[\scorebk(\Z_{k})\big] = \mfzero$ and existence of the $p \times p$ Fisher information matrix $\FJb_{\pintb,k} \defeq \Exp_{\pintb}\big[\scorebk(\Z_{k})\scorebk(\Z_{k})^\trans\big]$; see \citet[Theorem 7.2]{van2000asymptotic}. We do not require $\FJbpintk$ to be positive definite, since certain components of $\pintb$ may appear exclusively in specific arms; a situation formalized in Assumption~\ref{assm:DQM}(b). For instance, consider location models of the form $Z_{k,t}=\mu_k + \varepsilon_{k,t}$, where $\varepsilon_{k,t}$ are i.i.d.\ over $t$ with mean zero and fully known distribution. In this case, $p=K$ and $\pintb=(\mu_1,\dots,\mu_K)$. Note that for this simple example, Assumption~\ref{assm:DQM}(b) is indeed satisfied, and the $p\times p$ Fisher information matrices $\FJbpintk$ only have a nonzero element in the $(k,k)$-th position.
\end{remark}

The agent is allowed to update her sampling strategy at each time $t$ according to all the information available at that time. Formally, we define the filtration $(\filtr_t)_{t \geq 1}$ through
\begin{align*}
\filtr_t \defeq \sigma\left((\A_{s},\Y_{s}): s = 1,\dots,{t}\right),
\end{align*}
which collects the historical data on actions and rewards up to and including time $t$. The agent chooses the $(t+1)$-th action $A_{t+1}$ via a draw from a multinomial distribution, with probabilities denoted by $\pi_{t+1}$, conditional on $\filtr_t$. We impose the following on the sampling strategy. Recall that $k^*$ denotes the (unique) optimal arm.

\begin{assumption}\label{assm:sampling_scheme}
Let $\freq_{k,T} = \sum_{s=1}^T \indicator_{\{A_s = k\}}$ be the number of arm-$k$ pulls up to time $t$. For all $\pintb \in \Pintb$, we assume $(a)$ and either $(b)$ or $(b^*)$ below.
\begin{itemize}
\item[(a)] For all $t = 1, \ldots, T - 1$, the conditional sampling probabilities
\begin{align*}
\pi_{t+1}(k) 
\defeq \prob(A_{t+1} = k \cond  \filtr_t),
\quad k\in[K],
\end{align*}
do not depend on $\pintb$.
\item[(b)]
As $T\to\infty$ we have, for $k\neq k^*$,
\[
\frac{\freq_{k,T}}{\log T} \to C_k(\pintb) \in (0,\infty), ~~ a.s.
\]
\item[($b^*$)]
As $T\to\infty$ we have, for $k=1,\dots,K$,
\[
\frac{\freq_{k,T}}{ T} \to C_k(\pintb) \in (0, 1), ~~ a.s.
\]
\end{itemize}
\end{assumption}
\smallskip
\begin{remark} \label{remark:convergent_algorithms}
Note that Assumption~\ref{assm:sampling_scheme}(b) implies 
$ \freq_{k^*,T} / T \to C_{k^*}(\pintb)\defeq 1$ almost surely.
Define $\Delta_k=\Delta_k(\pintb) \defeq \mu_{k^*}(\pintb) - \mu_k(\pintb)$. 
Assumption~\ref{assm:sampling_scheme}(a)-(b)
is satisfied by, for example, the popular Gaussian Thompson sampling in \cite{thompson1933likelihood} and UCB1 proposed in \cite{auer2002finite}, imposing a known reward variance equal to $\sigma^2$, with $C_k(\pintb) = 2\sigma^2/ \Delta_k^2$; see \cite{fan2022typical}. The rate $\log(T)$ in Assumption~\ref{assm:sampling_scheme}(b) is commonly found as the rate with which suboptimal arms are pulled. Our results below can be easily adapted to other rates for the suboptimal arms, as long as they are $o(T)$.
\end{remark}
\begin{remark}
Randomized controlled trials (RCTs) are an example for which 
Assumption~\ref{assm:sampling_scheme}(a)-$(b^*)$ is trivially met. Adaptive sampling examples can arise, for example, by `clipping' a sampling scheme, i.e. by imposing 
$\pi_{t+1}(k \cond \filtr_t)\in [ \epsilon, 1 - \epsilon]$  for some $\epsilon>0$.
\end{remark}




\section{Local Asymptotic Normality} \label{sec:limitexperiment}
We establish Local Asymptotic Normality of the bandit experiment by, in Section~\ref{sec:LAQ}, establishing a quadratic expansion for the likelihood ratios and, subsequently, in Section~\ref{sec:LAN}, establishing asymptotic normality of that expansion.

\subsection{Quadratic expansion of likelihood ratios}\label{sec:LAQ}

Impose Assumptions \ref{assm:DQM}-\ref{assm:sampling_scheme}. 
Let $\pintb\in\Pintb$.
Set, for $j=1,\dots,p$,
\[
\rate_{j,T} = \rate_{j,T}(\pintb)=  
\begin{cases}
\sqrt{T}, &  \text{if } \FJb_{\pintb,k^*}[j,j] > 0; \\
\sqrt{s_T},  &   \text{else},
\end{cases}   
\]
where $s_T=\log(T)$ in case of Assumption~\ref{assm:sampling_scheme}(b) 
and $s_T=T$ in case of Assumption~\ref{assm:sampling_scheme}$(b^*)$.
Let $\Rate_T$ denote the diagonal $p \times p$ matrix with entries $\Rate_{T}[j,j] = \rate_{j,T}$. Then we
localize the parameter of interest at $\pintb$, with $\lpintb = (\lpint_1,\dots,\lpint_p)^\trans \in \mathbb{R}^p$, using
\begin{equation}
\pintb_T =  \pintb + \Rate_T^{-1} \lpintb.
\end{equation}
As $\Pintb$ is open we have $\pintb_T\in\Pintb$ for $T$ large.

Let $\lawn_{\pintb,\lpintb}$ denote the law of $(\A_{1},\Y_{1},\dots,\A_T,\Y_T)$ generated by the aforementioned stochastic $K$-armed bandit problem with parameter $\pintb_T$. Formally, we define the localized sequence of experiments as
\begin{align*}
\experimentn_{\pintb} \defeq \left(\Omega^{(T)}, \mathcal{F}^{(T)}, \left(\lawn_{\pintb,\lpintb} : \lpintb\in\rR^{p}\right)\right), ~~~ T\in\rN,
\end{align*}
where $\Omega^{(T)} = \left(\SK\otimes\rR\right)^T$ and $\filtrn = \mathcal{B}\left(\Omega^{(T)}\right)$, the Borel $\sigma$-field.

Using Assumption~\ref{assm:DQM} and  Assumption~\ref{assm:sampling_scheme}(a),  the log-likelihood ratio equals
\begin{equation*}
\begin{aligned}
    \log\frac{\dd\lawn_{\pintb,\lpintb}}{\dd\lawn_{\pintb,\mfzero}} 
&= \log\frac{\prod_{t=1}^T \pi_t(\A_t\cond\filtr_{t-1})\f_{A_t}(\Y_t\cond\pintb_T)}{\prod_{t=1}^T \pi_t(\A_t\cond\filtr_{t-1})\f_{A_t}(\Y_t\cond\pintb)}  
= \sum_{t=1}^T \log\frac{\f_{A_t}(\Y_t\cond\pintb_T)}{\f_{A_t}(\Y_t\cond\pintb)}  \\
&=
\sum_{k=1}^{K}\sum_{t=1}^T \indicator_{\{A_t = k\}}\log\frac{\fzk(\Z_{k,t}\cond \pintb + \Rate_T^{-1} \lpintb)}{\fzk(\Z_{k,t }\cond\pintb)} .
\end{aligned}
\end{equation*}
Note that the rate at which inference on a component of $\pintb$ is possible, is determined by the fastest rate among all arms whose reward distribution depends on that component. To make things precise, let $a_{k,T}=a_{k,T}(\pintb)\defeq \sqrt{s_T}$ for $k\neq k^*$ and $a_{k^*,T} =a_{k^*,T}(\pintb)\defeq \sqrt{T}$.
For $u\in\mathbb{R}^p$ and $k\in\SK$, introduce,
\begin{align*}
\Lambda_{\pintb,k}^{(T)}(u)  
 = \sum_{t=1}^T \indicator_{\{A_t = k\}}\log\frac{\fzk(\Zkt\cond \pintb + u / a_{k,T})}{\fzk(\Z_{k,t }\cond\pintb)}.
\end{align*}
We then notice that
\begin{align*}
    \log\frac{\dd\lawn_{\pintb,\lpintb}}{\dd\lawn_{\pintb,\mfzero}} 
    &=
    \sum_{k=1}^{K}\LLRn_{\pintb,k}( a_{k,T}\Rate_T^{-1}\lpintb). 
\end{align*}

\begin{proposition}\label{prop:likelihood_expansion}
Let Assumptions \ref{assm:DQM}--\ref{assm:sampling_scheme} hold and $\pintb\in\Pintb$.
And let $u_T$ be a bounded sequence in $\mathbb{R}^p$.
Under $\lawn_{\pintb,\mfzero}$, we have, for $k\in\SK$, the decomposition
\begin{align} \label{eqn:loglikelihoodratio_sequence}
\LLRn_{\pintb,k}(u_T)  
= u_T^\trans\CSb_{k,T} - \frac{1}{2}u_T^\trans\QVb_{k,T}u_T + \op(1), 
\end{align}
where
\begin{align*}
\CSb_{k,T} &\defeq \frac{1}{a_{k,T}(\pintb)} \sum_{t=1}^{T}\indicator_{\{A_{t} = k\}}\scorebk(\Y_{t}), \\
\QVb_{k,T} &\defeq \frac{1}{a_{k,T}^2(\pintb)} \sum_{t=1}^{T}\indicator_{\{A_{t} = k\}}\FJbpintk. 
\end{align*}
\end{proposition}

\begin{proof}[Proof of Proposition~\ref{prop:likelihood_expansion}]
We follow \citet[Proposition 1]{hallin2015quadratic} to prove the  expansion. To put notions in their language, we let $P_T = \lawn_{\mfzero}$,  define
\begin{align*}
\vS_{T t} =
 \frac{1}{a_{k,T}}\indicator_{\{A_t = k\}}\scorebk(Y_t)
=\frac{1}{a_{k,T}}\indicator_{\{A_t = k\}}\scorebk(\Zkt),
\end{align*}
for $t = 1,\dots,T$, and write the individual likelihood ratio of observation $t$ as
\begin{align*}
LR_{T t} = 1 + \indicator_{\{A_t = k\}}\left(\frac{\fzk(\Zkt\cond\pintb_T)}{\fzk(\Zkt\cond\pintb)}-1\right).
\end{align*}
By the DQM condition in Assumption~\ref{assm:DQM}, we can decompose 
\begin{align*}
\sqrt{LR_{T t}} = 1 + \frac{1}{2}u_T^\trans\vS_{T t} + \frac{1}{2}R_{T t},
\end{align*}
where $R_{T t} = \indicator_{\{A_t = k\}}\rate_k\big(\Zkt\cond  u_T/a_{k,T}\big)$.

We verify the four conditions in \citet[Proposition 1]{hallin2015quadratic} using, in their notation, the filtration defined by $\filtr_{T,t-1} = \sigma(\filtr_{t-1}, A_t)$.

Their \textit{Condition (a)} is trivially met by assumption. 

For their \textit{Condition (b)}, note 
\begin{align*}
\Exp_{P_T}\left[\vS_{T t}\cond \mathcal{F}_{T,t-1} \right] 
&= \frac{1}{a_{k,T}}\indicator_{\{\A_t = k\}}\,\Exp_{P_T}\big[\scorebk(\Zkt)\cond\A_t,\filtr_{t-1}\big]  
= \frac{1}{a_{k,T}} \indicator_{\{\A_t = k\}}\,\Exp_{P_T}\big[\scorebk(\Zkt)\big] 
= \mfzero,
\end{align*}
which yields their Display~(2). For $J_{T}$ in their Display~(3), under Assumption~\ref{assm:sampling_scheme}, we have 
\begin{align*} 
&  \sum_{t=1}^T \Exp_{P_T}\left[ \vS_{T t}\vS_{T t}^\trans \cond \filtr_{T,t-1}  \right]  
= \sum_{t=1}^T\indicator_{\{A_t = k\}}\Exp_{P_T}\left[ \vS_{T t}\vS_{T t}^\trans \cond \A_t,\filtr_{t-1} \right]  \\
&\qquad =\frac{1}{a_{k,T}^2}\sum_{t=1}^T \indicator_{\{A_t = k\}}\Exp_{P_T}\left[\scorebk(\Zkt)\scorebk(\Zkt)^\trans\right]  
= \FJbpintk \frac{\freq_{k,T}}{a_{k,T}^2}
= \Op(1).
\end{align*}
The conditional Lindeberg condition follows as, for any $\delta > 0$, 
\begin{align*}
&~ \sum_{t=1}^T \,\Exp_{P_T}\left[\big| u_T^\trans\vS_{T t}\big|^2\indicator_{\{|u_T^\trans\vS_{T t}| > \delta\}} \cond \filtr_{T,t-1}\right]  
= \frac{1}{a^2_{k,T}}\sum_{t=1}^T 
\indicator_{\{ A_t=k\}} \Exp_{P_T}\left[\big|u_T^\trans\scorebk(\Zkt)\big|^2\indicator_{\{|u_T^\trans\vS_{T t}| > \delta\}}\right]  \\
&\qquad = \frac{\freq_{k,T}}{ a^2_{k,T}}\times\Exp\left[\big|u_T^\trans\scorebk(Z_{k,1})\big|^2\indicator_{\{|u_T^\trans\scorebk(Z_{k,1})| > a_{k,T}\delta\}}\right] =O_P(1)\times o(1) =o_P(1).
\end{align*}

For their \textit{Condition~(c)}, note 
\begin{equation} \label{eqn:proof_condition(c)}
\begin{aligned}
& \sum_{t=1}^T \Exp_{P_T}\left[ R_{T t}^2 \cond \filtr_{T,t-1} \right] 
= \sum_{t=1}^T 1_{\{ A_t=k\}} \Exp_{P_T}\left[r_k^2\big(\Zkt\cond u_T /  a_{k,T} \big)\right] \\
&\qquad = \frac{\freq_{k,T}}{a_{k,T}^2} \times a_{k,T}^2  \Exp\left[r_k^2\big(Z_{k,1}\cond u_T/a_{k,T} \big)\right] = O_P(1) \times a_{k,T}^2 \times o(1/a_{k,T}^2) = o_P(1).
\end{aligned}
\end{equation}

Their Display (5) is satisfied as 
\begin{align*}
 &~ \sum_{t=1}^T (1- \Exp_{P_T}[LR_{T t}\cond\mathcal{F}_{T,t-1}]) = \sum_{t=1}^T -\Exp_{P_T}\left[\indicator_{\{A_t = k\}}\left(\frac{\fzk(\Zkt\cond\pintb_T)}{\fzk(\Zkt \cond \pintb)}-1\right)\cond \mathcal{F}_{T,t-1} \right]  \\
&\qquad = \sum_{t=1}^T - \indicator_{\{A_t=k\}} \Exp_{P_T}\left[\frac{\fzk(\Zkt\cond\pintb_T)}{\fzk(\Zkt\cond\pintb)}-1\right]  = 0,
\end{align*}
where the second equality follows the same arguments as (\ref{eqn:proof_condition(c)}). The last equality is automatic as the densities $\fzk$ are strictly positive.

Finally, their \textit{Condition~(d)} is naturally true under our setting. 
\end{proof}

\subsection{Local Asymptotic Normality}\label{sec:LAN}

The quadratic likelihood ratio expansion for each arm $k$ separately in Proposition~\ref{prop:likelihood_expansion} forms the basis of our LAN result for the bandit problem. Below, we combine the expansion for all arms and establish asymptotic normality.

To be precise, we introduce the $p$-dimensional \emph{central sequence}
\begin{equation*}
\CCSb_{T}[j] = 
\begin{cases}
\sum_{k=1}^K \CSb_{k,T}[j],  & \text{ in case of Assumption~\ref{assm:sampling_scheme}$(b^*)$; } \\
\CSb_{k^*,T}[j]  + \indicator_{\left\{\FJb_{\pintb,k^*}[j,j] = 0\right\}} \sum_{k\neq k^*} \CSb_{k,T}[j],  & 
\text{ in case of Assumption~\ref{assm:sampling_scheme}(b),}
\end{cases}    
\end{equation*}
for $j=1,\dots,p$, and the associated $p\times p$ \emph{Fisher-information} matrix
\begin{equation*}
\QQVb[\ell,m] = 
\begin{cases}
\sum_{k=1}^K  C_k(\pintb) \FJb_{\pintb,k}[\ell,m],  & \text{in case of Assumption~\ref{assm:sampling_scheme}$(b^*)$;} \\
 \FJb_{\pintb,k^*}[\ell,m]  + \indicator_{\left\{ \FJb_{\pintb,k^*}[\ell,m]=0\right\}} \sum_{k\neq k^*}  \FJb_{\pintb,k}[\ell,m],  & 
\text{ in case of Assumption~\ref{assm:sampling_scheme}(b),}
\end{cases}    
\end{equation*}
for $\ell, m = 1,\dots,p$.

\begin{proposition} \label{prop:LAN}
Let $\pintb\in\Pintb$ and Assumptions~\ref{assm:DQM}-\ref{assm:sampling_scheme} hold. Then we have,  under $\lawn_{\pintb,\mfzero}$,
\begin{align}
\QVb_{k,T} \stackrel{p}{\to} C_k(\pintb) \FJb_{\pintb,k}, 
\quad \text{for } k \in \SK, \label{eqn:convFI}
\end{align}
and 
\begin{align}
\begin{pmatrix}
\CSb_{1,T} \\
\vdots \\
\CSb_{K,T}
\end{pmatrix} 
\stackrel{d}{\to} 
\mathcal{N}\!\left(
\mfzero, 
\begin{pmatrix}
C_1(\pintb) \FJb_{\pintb,1} & \dots & 0 \\
\vdots & \ddots & \vdots \\
0 & \dots & C_K(\pintb) \FJb_{\pintb,K}
\end{pmatrix}
\right).  \label{eqn:convCS}
\end{align}
Moreover, for $\lpintb\in\mathbb{R}^p$ and still under $\lawn_{\pintb,\mfzero}$, we have
\begin{equation}\label{eqn:LAN}
\log\frac{\dd\lawn_{\pintb,\lpintb}}{\dd\lawn_{\pintb,\mfzero}} =
\lpintb^\prime\CCSb_{T} - \frac{1}{2}\lpintb^\prime \QQVb\lpintb + o_P(1).
\end{equation}
\end{proposition}

\begin{proof} All probabilities are evaluated under $\lawn_{\pintb,\mfzero}$. As the sequences $a_{k,T} \Rate_T^{-1} \lpintb$ are bounded
for $k\neq k^*$, 
Proposition~\ref{prop:likelihood_expansion} yields
\begin{align*}
\sum_{k\neq k^*} \LLRn_{\pintb,k}( a_{k,T}\Rate_T^{-1}\lpintb)
=  \sum_{k \neq k^*} \Bigg(a_{k,T} (\Rate_T^{-1}\lpintb)^\trans   \CSb_{k,T} - \frac{1}{2}a_{k,T}^2 (\Rate_T^{-1}\lpintb)^\trans\QVb_{k,T}  (\Rate_T^{-1}\lpintb) + \op(1)
\Bigg).
\end{align*}
For $k=k^*$ we cannot apply Proposition~\ref{prop:likelihood_expansion} directly: if Assumption~\ref{assm:sampling_scheme}(b) holds, the sequence $a_{k^*,T} \Rate_T^{-1} \lpintb$ might be unbounded (for Assumption~\ref{assm:sampling_scheme}($b^*$) there is actually no problem, but we include it over here as well for convenience). If we introduce for a $p$-vector $\mathbf{u}$, an accompanying vector $\tilde{\mathbf{u}}$ defined by $\tilde{\mathbf{u}}[j] = 0$ if $\FJb_{\pintb,k^*}[j,j]=0$ and $\tilde{\mathbf{u}}[j] = \mathbf{u}[j]$ otherwise, then Assumption~\ref{assm:DQM}(b) implies (i)
$
\LLRn_{\pintb,k^*}( a_{k^*,T}\Rate_T^{-1}\lpintb) =  \LLRn_{\pintb,k^*}( a_{k^*,T}\Rate_T^{-1}\tilde\lpintb),
$ (ii)  ${\mathbf{u}}^\prime \CSb_{k^*,T}=\tilde{\mathbf{u}}^\prime \CSb_{k^*,T} $ a.s., and (iii) ${\mathbf{u}}^\prime \QVb_{k^*,T} \mathbf{u} = \tilde {\mathbf{u}}^\prime \QVb_{k^*,T} \tilde{\mathbf{u}}$ a.s.
As the sequence $a_{k^*,T}(\pintb)\Rate_T^{-1}\tilde\lpintb$ is bounded (actually constant), we can apply Proposition~\ref{prop:likelihood_expansion} in combination with (i)--(iii) yielding
\[
\LLRn_{\pintb,k^*}( a_{k^*,T}\Rate_T^{-1}\lpintb) = 
\lpintb^\trans   \CSb_{k^*,T} - \frac{1}{2}   \lpintb^\trans\QVb_{k^*,T} \lpintb + \op(1).
\]

An obvious extension of Theorem~3.1 in \cite{melfipage2000} to $K>2$ arms yields, by Assumption~\ref{assm:sampling_scheme} and Slutsky's lemma,~(\ref{eqn:convFI}). A similar extension of their Theorem~3.2 yields, in combination with Assumption~\ref{assm:sampling_scheme} and Slutsky's lemma,~(\ref{eqn:convCS}).

In case of Assumption~\ref{assm:sampling_scheme}$(b^*)$, the LAN-property (\ref{eqn:LAN}) follows directly from the above. 
In case of Assumption~\ref{assm:sampling_scheme}$(b)$ we note that, for $k\neq k^*$, $a_{k,T}\Rate_T^{-1}\lpintb \to \bar{\lpintb}$ where $\bar{\lpintb}_j=\mfzero$ if $\FJb_{\pintb,k^*}[j,j]>0$ and $\bar{\lpintb}_j=\lpintb_j$ otherwise. Now the result follows.
\end{proof}

\section{Monte Carlo illustration}\label{sec:MC}


We consider a two-armed multi-armed bandit (MAB) setting ($K = 2$), where the potential outcomes for each arm $k = 1,2$ are generated as
\begin{align*}
\Zkt = \mu_k + \ekt,
\end{align*}
with innovations $\varepsilon_{k,t}$ that are i.i.d.\ Logistic with mean zero and scaled to have unit variance, independent across both $k$ and $t$. The parameter of interest is $\pintb = (\mu_1, \mu_2)$. We set $\mu_2 = 0$, $\mu_1 = m_1/\sqrt{T}$, and $T = 500$. All results are based on $50,000$ replications. 

Define the cumulative rewards $R_{k, t} = \sum_{s=1}^t\indicator_{\{A_{s} = k\}}\Y_{s} = \sum_{s=1}^t\indicator_{\{A_{s} = k\}}\Z_{k,s}$, for $k = 1,2$. We consider both algorithms mentioned in Remark~\ref{remark:convergent_algorithms}:
\begin{itemize}
    \item[-] \textbf{Gaussian Thompson sampling with prior $\mathcal{N}(0,1)$}: Conditional on the filtration $\mathcal{F}_t$, the posterior of $\mu_k$ is assumed to be $\mathcal{N}\left(R_{k,t}/(\freq_{k,t} + 1) , 1/(\freq_{k,t} + 1) \right)$, $k = 1,2$. The probability of choosing Arm-$2$ at round $t+1$ is $$\Phi\left(\left(\frac{R_{2,t}}{\freq_{2,t}+1} - \frac{R_{1,t}}{\freq_{1,t}+1}\right) \bigg/ \sqrt{\frac{1}{\freq_{1,t}+1}+\frac{1}{\freq_{2,t}+1}}\right).$$ 
    \item[-] \textbf{UCB1 sampling}: At round $t+1$, the algorithm selects the arm that maximizes the upper confidence bound $$\frac{R_{k,t}}{\freq_{k,t}} + \sqrt{\frac{2\log(t+1)}{\freq_{k,t}}}.$$
\end{itemize}
Note that we use \emph{Gaussian} Thompson sampling even if our reward distribution is Logistic. Such misspecification is allowed in the results of \cite{fan2022typical}. We use a Logistic reward distribution to prevent, in the simulations below, exact Gaussian distributions for the statistics of interest. After all, we want to study whether the limiting distributions provide good approximations to finite-sample distributions.

\begin{figure}[ht] 
\centering
\includegraphics[width = 6in]{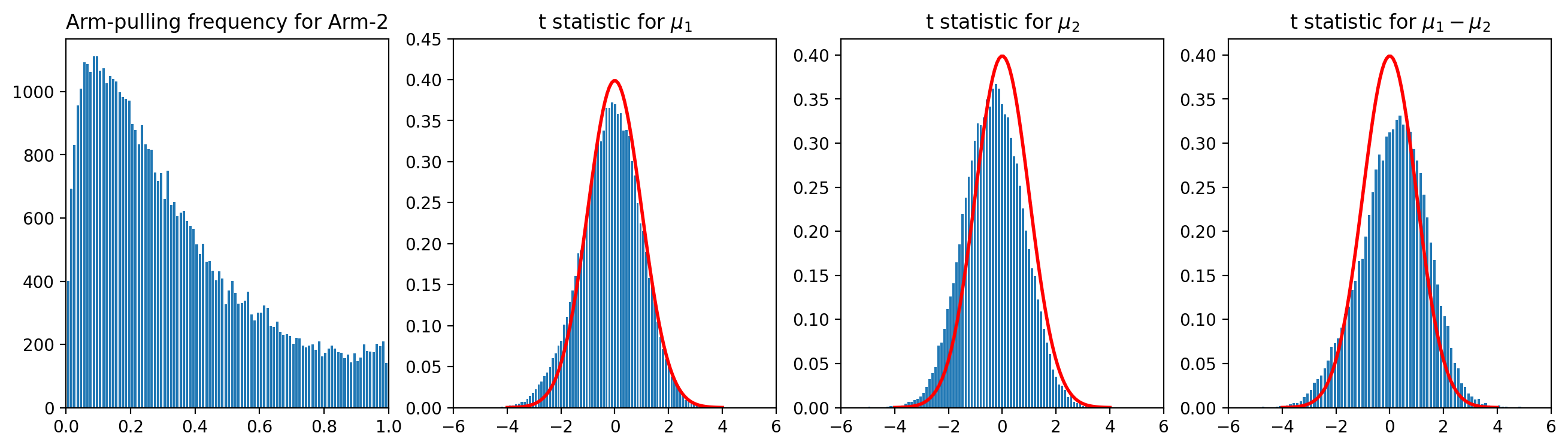}
\includegraphics[width = 6in]{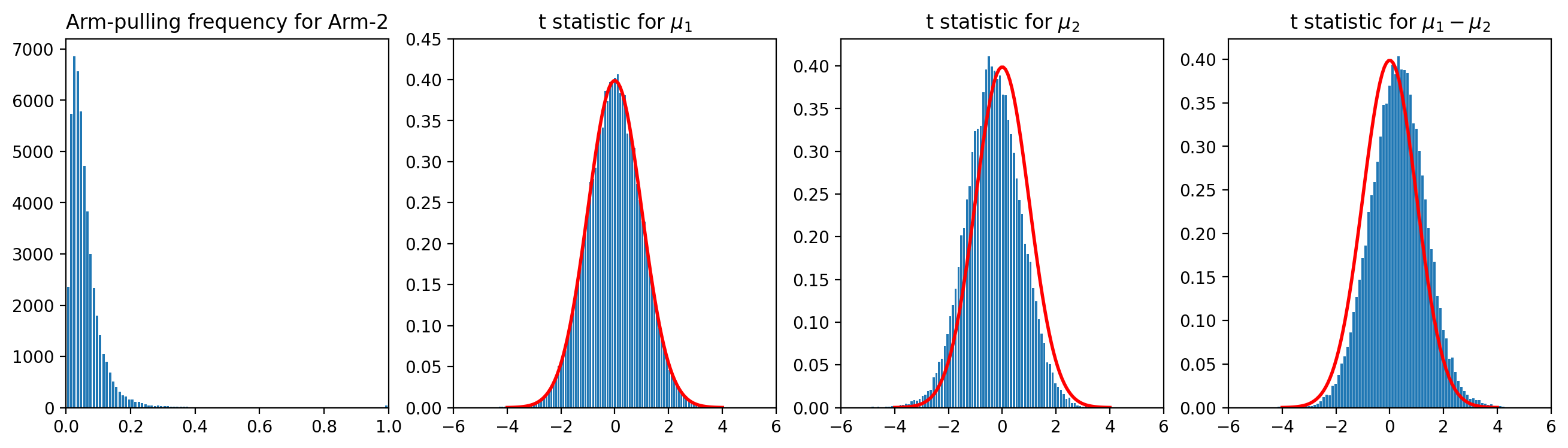}
\includegraphics[width = 6in]{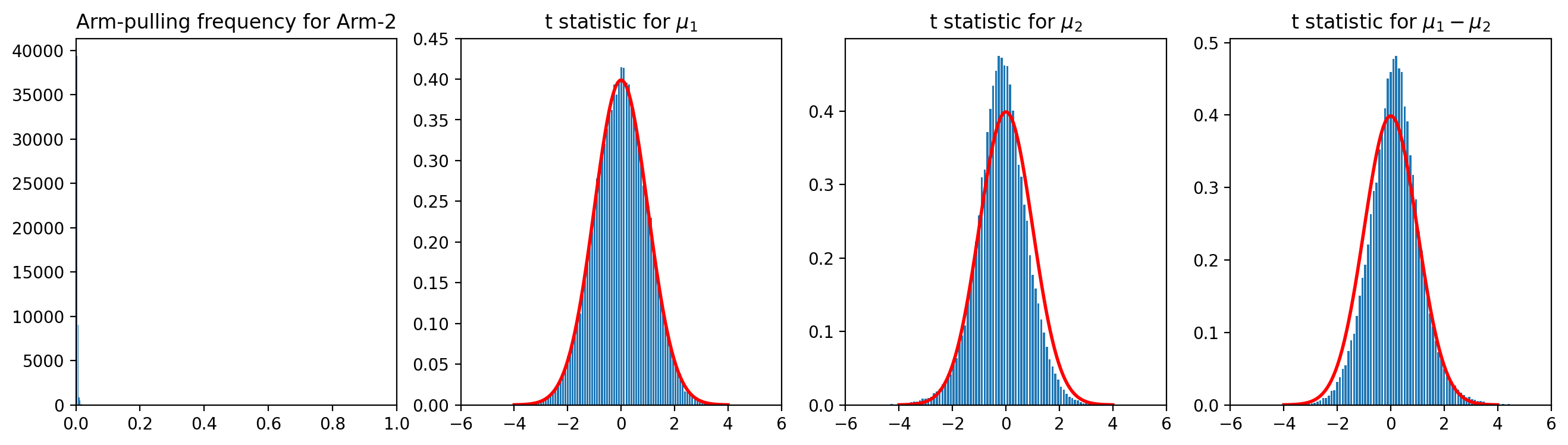}
\includegraphics[width = 6in]{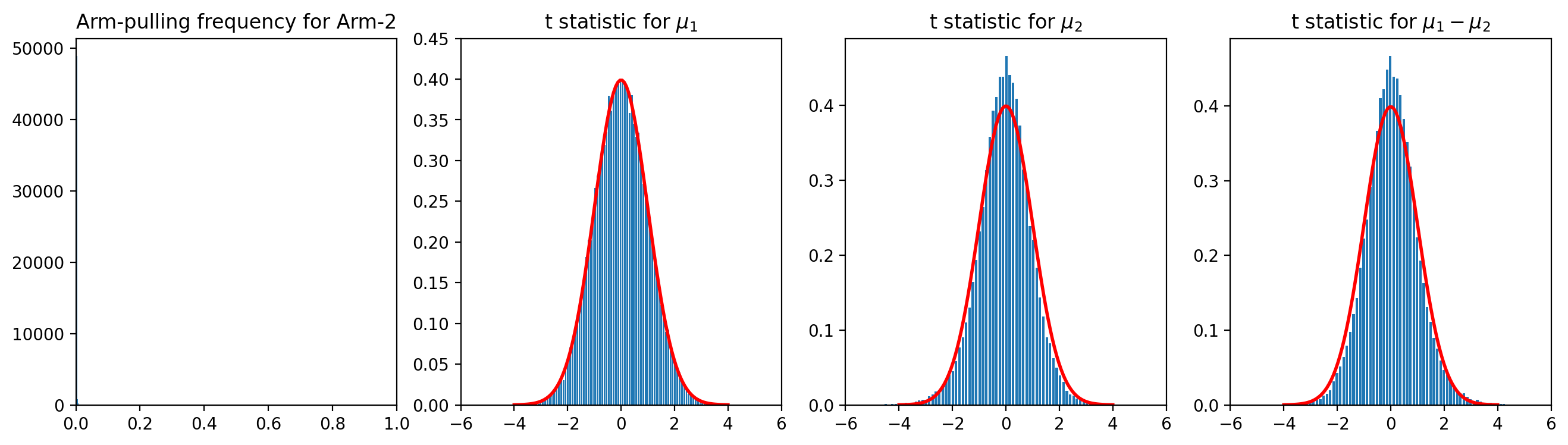}
\caption{{\small Histograms of the pulling frequency for the suboptimal Arm 2, followed by the $t$-statistics for $\mu_1$, $\mu_2$, and $\delta$ (from left to right), under \textit{Gaussian Thompson sampling}. The four panels from top to bottom correspond to four values of $m_1$: 2, 10, 50, and 75, respectively.}}
\label{fig:MABseq_statz_Thompson}
\end{figure}

\begin{figure}[ht] 
\centering
\includegraphics[width = 6in]{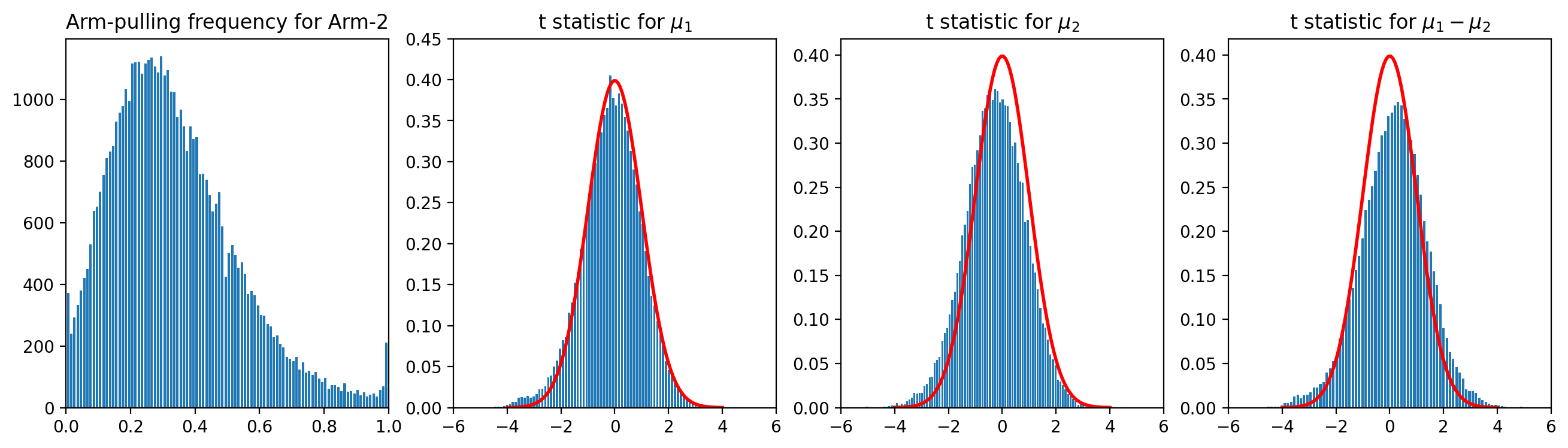}
\includegraphics[width = 6in]{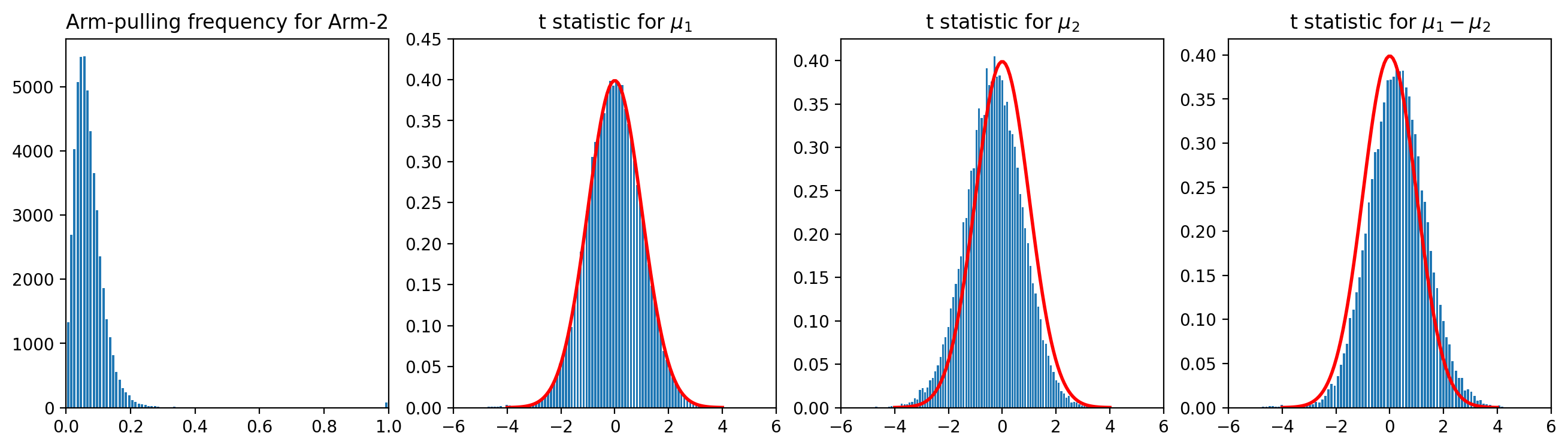}
\includegraphics[width = 6in]{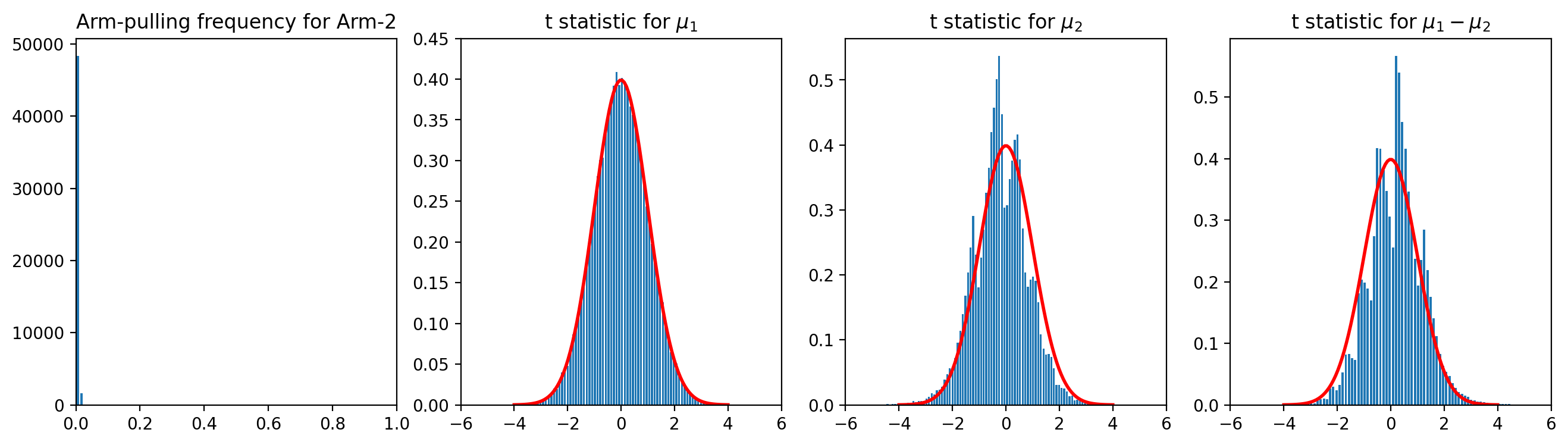}
\includegraphics[width = 6in]{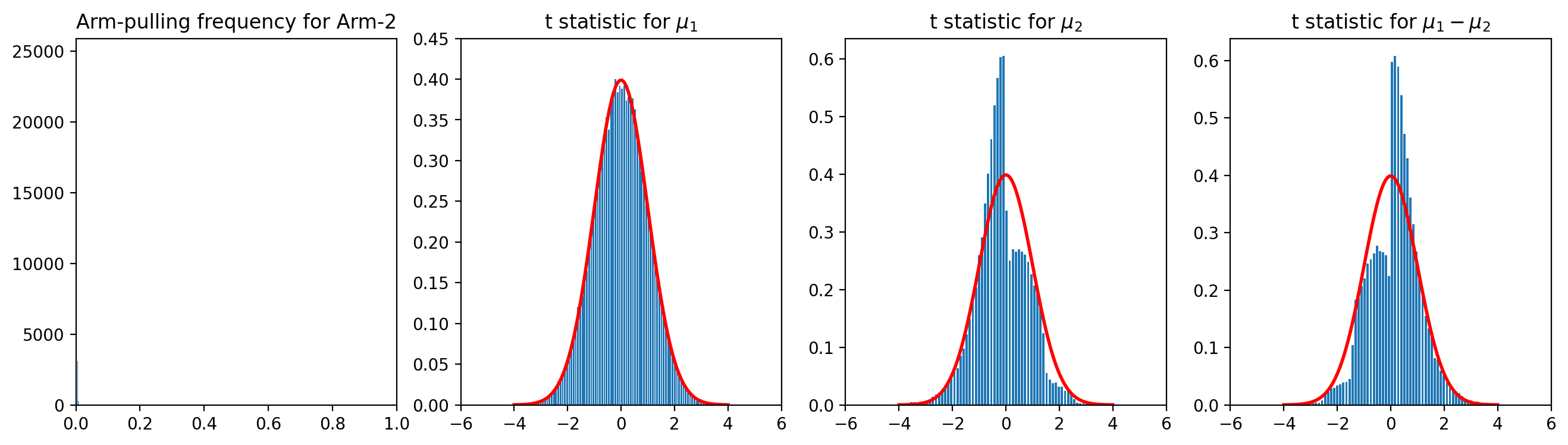}
\caption{{\small Histograms of the pulling frequency for the suboptimal Arm 2, followed by the $t$-statistics for $\mu_1$, $\mu_2$, and $\delta$ (from left to right), under \textit{UCB1 sampling}. The four panels from top to bottom correspond to four values of $m_1$: 2, 10, 50, and 75, respectively.}}
\label{fig:MABseq_statz_UCB1}
\end{figure}

In Figure~\ref{fig:MABseq_statz_Thompson}, we display, from left to right, the histograms of: (i) the arm-pulling frequency for Arm 2, $\freq_{2,T}$; (ii)–(iii) the classical Student’s $t$-statistics for $\mu_1$ and $\mu_2$, both defined as $$\stat^{\mu}_k \defeq \frac{ R_{k,T}/\freq_{k,T} - \mu_k}{\sqrt{1/\freq_{k,T}}},$$ and (iv) the $t$-statistic for the difference parameter $\delta \defeq \mu_1 - \mu_2$, given by $$\stat^{\delta} \defeq \frac{R_{1,T}/\freq_{1,T} - R_{2,T}/\freq_{2,T} - \delta}{\sqrt{1/\freq_{1,T} + 1/\freq_{2,T}}},$$ under Gaussian Thompson sampling. 

We experiment with four values of $m_1$: 2, 10, 50, and 75, shown from top to bottom. When the expected reward gap between the two arms is small ($m_1 = 2$, top panel), none of the $t$-statistics---those for $\mu_1$, $\mu_2$, or $\delta$---exhibits an approximate standard normal distribution. This non-standard asymptotic behavior can instead be characterized by the stochastic–differential–equation-based limit experiment developed under equal-arms asymptotics in \cite{zhou2025bandit}. When the gap becomes larger ($m_1 = 10$, second panel), the $t$-statistic for $\mu_1$ begins to approach a standard normal distribution, whereas those for $\mu_2$ and $\delta$ still clearly deviate from normality. This conclusion persists even when $m_1$ increases to 50 (third panel), where the gap is $\delta = 50/\sqrt{500} \approx 2.236$, and even to 75 (bottom panel)—a setting in which, in most replications, Arm 2 is pulled only once. In both cases, the $t$-statistics for $\mu_2$ and $\delta$ show no indication of converging toward normality.

Figure~\ref{fig:MABseq_statz_UCB1} serves as a counterpart of Figure~\ref{fig:MABseq_statz_Thompson} but under the UCB1 algorithm mentioned above. All conclusions continue to hold, except that the deviations from normality in the $t$-statistics for $\mu_2$ and $\delta$ become even more severe. These simulation results indicate that, although the limit experiment theoretically guarantees normality for the standard (test) statistics, the $\log(T)$ rate for the suboptimal arms is too slow for the theoretical limit to provide a reliable approximation in finite samples---even with a moderately large $T = 500$.

\bibliographystyle{asa}
\bibliography{references}


\end{document}